\documentclass[12pt,reqno]{amsart}
\usepackage[dvips]{graphics}
\usepackage{amssymb}
\usepackage{dsfont}

\numberwithin{equation}{section}

\newtheorem{thm}{Theorem}[section]
\newtheorem*{thm*}{Theorem}
\newtheorem*{thmmain*}{MAIN THEOREM}
\newtheorem{lem}[thm]{Lemma}
\newtheorem{cor}[thm]{Corollary}
\newtheorem{prop}[thm]{Proposition}
\newtheorem*{prop*}{Proposition}
 
\theoremstyle{definition}

\theoremstyle{remark}
\newtheorem{rem}{Remark}[section]

\newcommand{\tref}[1]{Theorem~\ref{#1}}
\newcommand{\cref}[1]{Corollary~\ref{#1}}
\newcommand{\pref}[1]{Proposition~\ref{#1}}

\newcommand{\lref}[1]{Lemma~\ref{#1}}

\newcommand{\rref}[1]{Remark~\ref{#1}}

\def\R{\mathds{R}}
\def\dim{\mathop{\text{dim}}}
\def\ind{\mathop{\text{ind}}}

\def\Jac{\mathop{\text{Jac}}}
\def\Lagr{\mathop{\text{Lagr}}}
\def\Sym{\mathop{\text{Sym}}}

\begin{document}
\title{Notes on the Jacobi equation}

\author{Alexander Lytchak}
\address{A. Lytchak, Mathematisches Institut, Universit\"at Bonn,
Beringstr. 1, 53115 Bonn, Germany}
\email{lytchak\@@math.uni-bonn.de}



\begin{abstract}
We discuss some properties of  Jacobi fields  that do not involve
assumptions on the curvature endomorphism. We
compare indices of different  spaces of
Jacobi fields and  give some  applications to  Riemannian geometry.
\end{abstract}

\thanks{The author was supported in part by the SFB 611 
{\it Singul\"are Ph\"anomene und Skalierung in mathematischen Modellen.}}

\maketitle
\renewcommand{\theequation}{\arabic{section}.\arabic{equation}}

\pagenumbering{arabic}

\section{Introduction}
This note is  a collection of results about conjugate
points of Jacobi fields for which we could not find an appropriate
reference in the literature, while we were working on 
conjugate points in quotients of Riemannian manifolds.
We discuss  here some basic results  about indices of spaces
of Jacobi fields that do not involve Morse theory
nor   comparison  results. For the latter the reader can consult  
\cite{eschen} or any text book
on Riemannian geometry, for instance \cite{Sakai}.

Let $V$ be an $m$-dimensional Euclidean vector space and let $R(t), t\in  I$
be a smooth family of symmetric endomorphisms defined on an interval 
$I \subset \mathbb R$.  The equation $$Y'' (t) +R(t) Y(t)=0$$ is called
the {\it Jacobi equation} defined by $R(t)$. Solutions of the Jacobi
equations are called {\it Jacobi fields}.  By $\Jac$ we will denote 
the vector space of all Jacobi fields. For any $t \in I$ we have an 
identification $I^t: \Jac \to V\times V$, given by $J\to (J(t),J' (t) )$.
 On $\Jac$  the symplectic form 
$\omega (J_1,J_2) = \langle J_1 (t) , J_2 ' (t) 
\rangle  - \langle J_ 1 ' (t) , J_2 (t) \rangle$
is independent of $t$, due to the symmetry of $R(t)$.
For any $t$, this symplectic form corresponds to the canonical
symplectic form on $V\times V$ via the identification $I^t$.

 For a subspace $W\subset \Jac$ we denote by $W^{\perp}$ the orthogonal
complement of $W$ with respect to $\omega$. The subspace $W$ is called
{\it isotropic} if $W\subset W^{\perp}$; and it is called {\it Lagrangian}
if $W=W^{\perp}$. 

Let $W$ be an isotropic subspace of $\Jac$. For $t\in I$ we set
$W(t) = \{ J(t) | J\in W \}$ and $W^t =\{ J\in W | J(t)=0 \}$.
We say that $t$ is {\it $W$-focal} if $\dim (W^t)= \dim (W)- \dim (W(t))>0$ and
call this number the {\it $W$-focal index} of $t$. This number  will be
denoted by $f^W(t)$.   For a  subinterval $I_0\subset I$, we define the
{\it $W$-index} of $I_0$ to be $\ind _W (I_0) = \Sigma _{t\in I_0} f^W(t)$.

 In Riemannian geometry one mostly considers the indices of special 
Lagrangians defined by some submanifolds (see \cite{Sakai} and
Subsection \ref{inter} below).  Here we emphasize a more abstract point
of view that involve all Lagrangian subspaces and, more general,
isotropic subspaces of the space of all Jacobian fields. For the
natural appearance of such situations one should look at \cite{Wilk},
the paper in which the  important tool of transversal Jacobi equation was
invented.

 Now we  can state our results, that seem to be known to the experts 
in many special cases. 
\begin{thm} \label{theorem}
Let $\Lambda _1, \Lambda _2 \subset \Jac$ be any Lagrangians.
Then for any interval $I_0 \subset I$ we have 
$\ind _{\Lambda _1} (I_0) -\ind _{\Lambda _2} (I_0) \leq \dim (V)$.
\end{thm}

See \pref{gen}, for a slightly more general statement.  As a 
consequence of \tref{theorem} we deduce:
\begin{cor} \label{complete}
 Let $M$ be a  Riemannian manifold without conjugate points. Then
for any submanifold $N$ of $M$ and any geodesic $\gamma$ orthogonal to
$N$ there are at most $\dim (N)$ focal points of $N$ along $\gamma$ (counted
with multiplicity).  
\end{cor}

 Another direct consequence is a non-geometric proof of the following
well known differential geometric result (see Subsection \ref{conjpoit} for
the definition of conjugate points): 
\begin{cor} \label{con}
Let $V,R,\Jac$ be as above. If for some $a<b \in I$ the points $a$ and
$b$ are conjugate, then for each $\bar a \leq a$ there is some 
$\bar b \in [a,b]$ that is conjugate to $\bar a$.
\end{cor}

 Another important issue for which we could not find a reference is 
the following semi-continuity and continuity statement. For  similar
continuity statements in the more general context of semi-Riemannian geometry
the reader should consult \cite{piccione}.

\begin{prop} \label{propo}
 Let $R_n (t)$ be a sequence of families of
symmetric endomorphisms converging in the $\mathcal C^0$ topology to $R(t)$.
Let $W_n$ be isotropic subspaces of $R_n$-Jacobi fields that converge
to an isotropic subspace $W$ of $R$-Jacobi fields.
Let $I_0 =[a,b]\subset I$ be a compact interval and assume that 
$f^{W_n} (a) =f^W(a)$ and $f^{W_n} (b)= f^W (b)$, for all $n$
large enough. Then 
 $\ind _W (I_0) \geq \ind _ {W_n} (I_0)$, for all $n$ large enough.  
If all $W_n$ are Lagrangians then this inequality becomes  an equality.
\end{prop}

 We prove \tref{theorem} and its Corollaries 
by using the continuity principle above  and by
reducing the claim to the $1$-dimensional situation with the help
of Wilkings transversal Jacobi equation. The proof involves the decomposition
of the index in the sum of the vertical and the horizontal indices with
respect to an isotropic subspace (\lref{inddeco}).
 See Subsection \ref{inter1} and \cite{Expl} 
for a geometric interpretation of these notions. We would like to
mention that  the non-continuity 
of indices for isotropic non-Lagrangian spaces and the
decomposition formula  of \lref{inddeco} cause  strange non-continuous
behavior of indices in quotients of Riemannian manifolds 
(\cite{Expl} and \rref{remark}).

In Section \ref{Sec2} we discuss basic facts about Jacobi fields, prove the
semi-continuity part of \pref{propo} and recall the arguments of \cite{Dui}
that relate the index to Lagrangian intersections and imply  the continuity 
part of \pref{propo}. In Section \ref{Sec3} we recall the construction
of the transversal Jacobi equation, due to Wilking. In Section \ref{Sec4}
we prove the remaining results.

I would like to thank Gudlaugur Thorbergsson for fruitful discussions.
I am grateful  to   Paulo Piccione for helpful remarks and for 
the  reference  \cite{piccione}.

\section{Semi-continuity and continuity of indices} \label{Sec2}
\subsection{Semi-continuity} We start with the only simple
comparison result that will be used in the sequel.

\begin{lem}
 Let $V,R,\Jac$ be as in the introduction. Assume that $||R(t)||$ is
bounded  above by $C^2 \in \R$. Let $J\in \Jac$ be a Jacobi field
with $J(t^- ) =0$. Then for all $t^+ \in I$ with
$|t^+ -t^-| < \frac 1 {2C}$ we have 
$|| J(t^+) - (t^+-t^-) \cdot J' (t^+)|| \leq C \cdot|| J' (t^+)||\cdot 
(t^+ -t^-)^2$.
\end{lem} 

\begin{proof}
We may assume $t^+ >t^-$.
From Rauch's comparison theorem (\cite{Sakai},p.149)  we deduce
$$||J(t)|| \leq \frac 1 {2C}  e^{C |t-t^-|} \cdot || J' (t^-)|| \leq
\frac 1 C || J' (t^-)||$$ for all $t\in [t^-,t^+]$. Thus
$||J'' (t)|| \leq C ||J' (t^-)||$, for all $t\in [t^-,t^+]$. 
Hence
$$||J' (t^+)|| \geq ||J' (t^-)|| - C\cdot ||J' (t^-)|| \cdot |t^+ -t^-|
\geq \frac 1 2  || J' (t^-)||$$

 Due to the Taylor formula, we find some $t\in [t^-,t^+]$ with
$$ || J(t^+) - (t^+-t^-) \cdot J' (t^+)||\leq \frac 1 2 || J'' (t)||  \cdot  
(t^+ -t^-)^2$$
The desired estimate now follows from 
 $$\frac 1 2  || J'' (t)|| \leq \frac 1 2  C \cdot ||J' (t^-)|| 
\leq  C \cdot || J' (t^+)||$$
\end{proof}

 Let again  $C^2\in \mathbb R$ be an upper bound for $||R(t)||$. Let $W$
be an isotropic subspace of $\Jac$ and let $t^+  >t^-$ be
$W$-focal points, with $|t^+- t^- | < \frac 1 {2C}$.
 Choose any $J_+ \in W^{t^+} $ and $J_- \in W^{t^-}$. 
Since $W$ is isotropic, we have $\langle J_{-} (t^+),J_{+}' (t^+) \rangle =0$.
From the last lemma we obtain now
$$\langle J_{-} ' (t^+) , J_{+} ' (t^+) \rangle  \leq  
 C \cdot  |t^+-t^-| \cdot ||J'_{-} (t^-)|| \cdot  ||J'_{+} (t^+)||$$

Thus $J^+$ and $J^-$ are almost orthogonal with respect to
the scalar product $s_{t+}$ on $\Jac$ defined by
$$s_{t^+} (J_1,J_2) := \langle J_1 (t^+), J_2 (t^+) \rangle +
\langle J_1 ' (t^+) , J_2 ' (t^+) \rangle$$
This has the following  consequences (cf. \cite{Sakai},p.61,p.101):

\begin{lem}
Let $W$ be an isotropic subspace of $\Jac$. Then the $W$-focal points
are discrete in $I$.  Moreover, there is some number $\epsilon$, that depends
only on an upper bound of $||R(t)||$, such that for an interval $I_0$ of
length $\leq \epsilon$ the inequality $\ind _W  (I_0 ) \leq \dim (W)$ holds. 
\end{lem}

 In the case $\dim (W) =1$ we get:
\begin{lem} \label{dim1}
Let $J$ be a non-zero Jacobi field. 
If $J(t^+)=J(t^-)=0$, for some $t^+>t^- \in I$,
then $|t^+-t^-| > \epsilon$, where $\epsilon$ depends only on the upper
bound on $||R(t)||$.
\end{lem}

  Finally we get:
\begin{lem} \label{semi}
 Let $R_n (t)$ be a sequence of families of
symmetric endomorphisms converging in the $\mathcal C^0$ topology to $R(t)$.
Let $W_n$ be isotropic spaces of $R_n$-Jacobi fields that converge
to an isotropic space $W$ of $R$-Jacobi fields.
Let $I_0 =[a,b]\subset I$ be a compact interval. Then
 $\ind _W (I_0) \geq \ind _ {W_n} (I_0)$, for all $n$ large enough.  
\end{lem}

\begin{proof} It is enough to observe that for $t_n \to t \in I_0$ the limit
of $W^{t_n}$ is contained in $W^t$, and that for sequences  $t^+_n  >t^-_n$
converging to the same $t\in I_0$, the limits of $W^{t_n ^+} $ and $W^{t_n^-}$
are orthogonal with respect to the scalar product $s_t$.
\end{proof}

\subsection{Continuity}  If the isotropic subspaces $W_n$
and $W$ in \lref{semi} are Lagrangian then the inequality turns out to be
an equality under the additional assumption, that the focal indices at the
boundary points are constant.
  To see this one has either to interpret the $W$-index 
as the index of 
some bilinear form (as it is done in the most important geometric situations, 
cf. \cite{Sakai},p.99), or
to interpret the index as the Maslov index of Lagrangian intersections,
as  in \cite{Dui},p.180-186, see also \cite{piccione} for a more detailed 
account.
  We are going to sketch the last approach for the 
convenience of the reader.

 Namely, the map $J\to (J,J')$ identifies Jacobi fields with flow
lines of the time dependent vector field $X'(t) =A(t) X(t)$ on
the vector space $T=V\times V$, where $A(t)$ is given by
$A(t) (v_1,v_2)= (v_2, -R(t)v_1)$. This flow preserves the canonical
symplectic form $\omega$ on $T$, given by 
$\omega ((v_1,v_2),(w_1,w_2))= \langle v_1,w_2 \rangle - 
\langle v_2,w_1 \rangle  $.
Thus, for each Lagrangian subspace $\Lambda \subset T$ the family 
$\Lambda (t):= \{X(t)| X\in \Lambda \}$ is a curve in the space
$\Lagr$ of all Lagrangians of $T$.

 Consider the fixed Lagrangian subspace $\Lambda _0 =\{ 0 \} \times V$
of $T$. By definition, for each Lagrangian subspace $\Lambda$ of $\Jac$,  the
focal index $f^{\Lambda } (t)$ is   given by 
$f^{\Lambda (t) } = \dim (\Lambda (t)  \cup \Lambda _0)$.

The space $Lagr ^0$ of all Lagrangians transversal to $\Lambda _0$
is a contractible space. Thus each curve  $\gamma (t)$ in $\Lagr $
whose endpoints are in $\Lagr ^0$  can be (uniquely up to homotopy),  
completed  to a closed curve $\bar \gamma$
by connecting the endpoints of $\gamma$
inside of $\Lagr ^0$.  Hence, such  $\gamma$ gives us  a well defined
element in $\pi _1 (\Lagr )= \mathbb Z$.  The image of such a curve
$\gamma$ in $\mathbb Z$ is called the Maslov-Arnold  index of $\gamma$
and is denoted by $[\gamma ]$.

 The Maslov-Arnold index  $[\gamma ]$ is equal to the intersection
number of $\bar \gamma$ and the cycle given by $\Lagr \setminus \Lagr ^0$
and can be computed as follows.  For each  time  $t$
with  non-zero  intersection  $F_t =\gamma (t) \cap \Lambda _0$, one
computes the restriction  to $F_t$ of the  symmetric bilinear form 
$B \in \Sym (\Lambda _0)$, given  by 
$\gamma ' (t) \in T_{\Lambda _0}  \Lagr = \Sym (\Lambda _0)$. If this
bilinear form $B$ on $F_t$ is non-degenerate, its signature is the contribution
of the point $\gamma (t)$ to the Maslov-Arnold index.  In our case,
$\gamma (t) = \Lambda (t)= \{X(t)| X\in \Lambda , X'(t)=A(t) X(t)\}$,
the bilinear form $B =\Lambda ' (t)$  is defined by $B(x,y)= \omega (A(t)x,y)$,
for $x,y \in \Lambda (t)$.  By the definition of $A(t)$, we have
$B(x,x)= ||x||^2$ for each $x\in \{0\} \times V= \Lambda _0$.
Thus $B$ is positive definite on each intersection space $F_t$
and the contribution of the $\Lambda$-focal point $t$ to $[\gamma ]$ is 
precisely the focal index $f^{\Lambda} (t)$. We conclude that
the Maslov-Arnold index of the flow line $\Lambda  :[a,b] \to \Lagr$
coincides with the $\Lambda$-index $\ind ^{\Lambda} ([a,b])$ if
   the endpoints $a$ and $b$ are not $\Lambda$-focal. Since
the Maslov-Arnold index $[\gamma ]$ is a topological notion, it is stable under
small perturbations and we get the same conclusion for 
$\ind ^{\Lambda}  ([a,b])$. Now we can finish the

\begin{proof}[Proof of \pref{propo}] 
  The semi-continuity in the general case was shown in \lref{semi}.
Thus let us assume that $W$ and $W_n$ are Lagrangian.
Due to the semi-continuity of indices and the assumption 
$f^{W_n} (a) =f^W(a)$ and $f^{W_n} (b)= f^W (b)$, we find some $\epsilon >0$
such that $W$ and $W_n$ have no focal points in $[a-\epsilon,a)$ and
$(b,b+\epsilon]$. Thus the $W$-indices of $[a,b]$ and of 
$[a-\epsilon ,b+\epsilon]$ coincides and the same statement is true for the
$W_n$-indices. Now the $W_n$- and $W$-indices of $[a-\epsilon,b+\epsilon]$
are equal to the corresponding Maslov-Arnold indices and the last ones 
are stable under small perturbations.  
\end{proof}

\subsection{Geometric interpretation} \label{inter}
Let $I$ be an interval and let $T\to I$ be a Riemannian vector bundle 
with a Riemannian connection  and a family of  symmetric 
endomorphisms $R:T\to T$.  The connection  is flat and defines
an isomorphism of $T$ and the canonical bundle $I\times V\to I$. Thus
all results discussed above apply to this situation. The most prominent
example is the case of the normal bundle $\mathcal N$ along a geodesic 
$\gamma$ in a Riemannian manifold $M$, where the endomorphisms $R$
are the curvature endomorphisms $R(X)= \mathcal R(X,\gamma ')\gamma '$.
Most prominent examples of Lagrangian subspaces of the spaces of Jacobi
fields are spaces $\Lambda ^N$ of all normal $N$-Jacobi fields, where $N$ is a
submanifold of $M$ orthogonal to $\gamma$. A special and most important
case is that of a $0$-dimensional submanifold $N =\{ \gamma (a) \}$ 
that defines
the Lagrangian $\Lambda ^a$ of  all Jacobi fields $J$ with $J(a)=0$.
In this case the $\Lambda$-index
can be interpreted as the index of a symmetric bilinear form on a
 Hilbert space or as the index of a Morse function on a space of curves.
In this case the results discussed above  are contained in any 
book on Riemannian geometry.

\section{Transversal Jacobi equation} \label{Sec3}
\subsection{The construction}
Let $T\to I$ be a Riemannian vector bundle over an interval $I$
with a Riemannian connection $\nabla$ and a field $R:T\to T$ of 
symmetric endomorphisms. Let $\Jac$ be the space of Jacobi fields and
let $W\subset \Jac$ be an isotropic subspace. We are going 
to describe Wilking's construction
of the transversal Jacobi equation (\cite{Wilk},p.3).

  Wilking observed that the family 
$\tilde W(t):= W(t) \oplus \{ J' (t)|J\in W ^t \}$ is a smooth subbundle
of $T$. Notice that $W(t)=\tilde W(t)$ for all non-focal values of $t$.
Denote by $\mathcal H$ the orthogonal complement of 
$\tilde W$ and by $P:T\to \mathcal H$  the orthogonal projection.
Then $P$ defines an identification between $\mathcal H$ and
$T/\tilde W$.  The mapping $A(J(t))=P(J' (t))$ extends to a smooth field
of homomorphisms $A:\tilde W \to \mathcal H$, by setting 
$A(J' (t))=0$, for all
$J\in W^t$. 

Consider the field  $R^{\mathcal H} :\mathcal H\to \mathcal H$ of symmetric
endomorphisms defined by $R^{\mathcal H} (Y)= P(R(Y)) + 3 AA^{\ast} (Y)$.
Denote by $\nabla ^{\mathcal H}$ the induced covariant derivative on
$\mathcal H$, that is defined by $\nabla ^{\mathcal H} (Y) =P (\nabla Y)$.
Wilking  proved (\cite{Wilk},p.5) that for each Jacobi field  
$J\in W ^{\perp} \subset \Jac$, the projection $Y=P(J)$ is an 
$R^{\mathcal H}$-Jacobi field, i.e., we have
$$ \nabla ^{\mathcal H} (\nabla ^{\mathcal H} (Y)) + R^{\mathcal H} (Y) =0$$ 
Two $R$-Jacobi fields $J_1,J_2 \in W^{\perp}$ have the same
projection to $\mathcal H$ if and only if $J_1 -J_2 \in W$. 
Thus the induced map $I: W^{\perp} /W \to \Jac ^{R^{\mathcal H}}$ is 
injective and by dimensional reasons it is an isomorphisms.
Thus $R^{\mathcal H}$-Jacobi fields are precisely the projections
of Jacobi fields in $W^{\perp}$; and Lagrangians in $\Jac ^{R^{\mathcal H}}$
are projections of Lagrangian in $\Jac$ that contain $W$.

 Finally, we have the following equality of indices:
\begin{lem} \label{inddeco}
In the notations above, for each Lagrangian subspace 
$\Lambda \subset \Jac$ that contains $W$ we have the equality 
$\ind _W (I) + \ind _{\Lambda /W} (I) =\ind _{\Lambda } (I)$. 
\end{lem}

\begin{proof}
Let $t\in I$ be given. For each $J_1 \in \Lambda$ and $J_2 \in W^t$,
we have $\langle J_1 (t) , J'_2 (t)\rangle =0$. Thus for
each $J\in \Lambda$ the inclusions $J(t)\in W(t)$ and $J(t)\in \tilde W(t)$
are equivalent.   Hence
$\Lambda (t) \cap \tilde W (t) =\Lambda (t) \cap W(t)$
and we deduce 
$f^t (W) +f^t(\Lambda /W)= f^t(\Lambda )$. Summing up the
focal indices gives us the result.
\end{proof}

\begin{rem} \label{remark}
The index formula above  imply the following explosion of indices in quotients,
see the next subsection for a geometric interpretation. In the notations
of \pref{propo}, let $W_n \subset \Lambda _n$ be pairs of $R_n$-isotropic
and larger Lagrangian subspaces that converge to the pair $W\subset \Lambda$.
We get smooth transversal Jacobi equation with symmetric endomorphisms
$R^{\mathcal H_n}$ and  $R^{\mathcal H}$.
Note that at all non-focal points of $W$ the transversal endomorphisms
$R^{\mathcal H_n}$ converge to $R^{\mathcal H}$  and $\Lambda _n / W_n$
converge to $\Lambda /W$ on the complement of the set of $W$-focal points.
However, if $\ind _{W_n} (I_0) < \ind _   W(I_0)$ for all $n$,
 a situation that happens 
very often, then we deduce $\ind _{\Lambda /W} (I_0) < 
\ind _{\Lambda _n  /W_n} (I_0)$. Thus at some $W$-focal points the 
transversal endomorphisms $ R^{\mathcal H_n}$ are forced to have a very 
steep bump producing focal points.  It seems that if indices of $W_n$ are 
stable, the fields  $   R^{\mathcal H_n}$ should converge to $R^{\mathcal H}$
in the $\mathcal C^0$ topology, but we have checked this statement
only in the  special situation of \cite{Expl}. 
\end{rem}

\subsection{Geometric interpretation}  \label{inter1}
In the situation of \cite{Wilk}, the meaning of $\Lambda /W$
is not easy to describe. However, the origin of the tensor $R^{\mathcal H}$ 
defined above is the curvature endomorphism in the base of a Riemannian 
submersion, a situation that we will shortly describe now (cf. \cite{LTvar}
for a more detailed exposition).
Thus let $f:M\to B$ be a Riemannian submersion, let $\gamma $ be a geodesic
in $M$ and let $\bar \gamma = f(\gamma )$ be its image in $B$. Let
$R,\bar R$ be the curvature endomorphisms along $\gamma$ and 
$\bar \gamma$ respectively.   Consider the space $W$ of all Jacobi
fields along $\gamma$ that arise as variational fields of geodesic variations
$\gamma _s$ such that $f(\gamma _s) =\bar \gamma$, for all $s$.  Then
$W$ is an isotropic subspace, since it is contained the space $\Lambda ^N$
of normal $N$-Jacobi fields, where $N=f^{-1} (f(\gamma (a)))$ for any
$a$.  In this case the additional term $AA^{\ast}$ is just the O'Neill
tensor (\cite{o},p.465) and  the field  $R^{\mathcal H}$ coincides with 
$\bar R$.  In this case $W ^{\perp}$ 
consists of all variational fields of variations
through horizontal geodesics, as one deduces by  counting of dimensions.
The ``horizontal'' index $\ind _{\Lambda /W} (\gamma )$ describes the index of 
the geodesic $\bar \gamma$ in the quotient space. The ``vertical'' index
$\ind _W (\gamma )$ is $0$ in this case, but in the similar and much 
more general
situation of a singular Riemannian foliation (cf. \cite{Expl})
 it counts the intersections
of $\gamma$ with singular leaves. Then the formula of \lref{inddeco} describes
a natural decomposition of the $\Lambda$-index in a horizontal part 
seen in the quotient below and a vertical part counting the intersections with
the singular leaves.

\section{Applications} \label{Sec4}
\subsection{Conjugate points} \label{conjpoit}
Let $V,I\subset \R,R(t),\Jac$ be as in the introduction. Points $a<b \in I$
are called {\it conjugate} if there is some $J\in \Jac$ with $J(a)=J(b)=0$.
Equivalently, one can say that $b$ is $\Lambda ^a$-focal.
Here and below we use the notation $\Lambda ^a = \{ J\in \Jac | J(a)=0 \}$.
Before proving \tref{theorem} we are going to derive
its consequences \cref{complete} and \cref{con}.

\begin{proof} [Proof of \cref{con}]  Assume the contrary. Then
the Lagrangian $\Lambda ^{\bar a}$ has index $0$ on the interval $I_0=[a,b]$
and $\ind _{\Lambda ^a} (I _0) \geq dim (V)+1$. 
This contradicts \tref{theorem}.
\end{proof}

\begin{proof} [Proof of \cref{complete}]
Consider the normal bundle $\mathcal N$ along $\gamma$ with
the induced connection $\nabla$ and the curvature endomorphism $R$. 
Let $\Lambda ^N$ denote the Lagrangian of all normal $N$-Jacobi fields
along $\gamma$. Then the number of $N$-focal points along $\gamma$ counted
with multiplicity is precisely $\ind_{\Lambda ^N} (I) -f^{\Lambda ^N} (0) =
\ind_{\Lambda ^N}(I) - ((n-1) - \dim (N))$, where $I$ is the
interval of definition of $\gamma$.  

Thus it is enough to prove $\ind_{\Lambda ^N} (I) \leq n-1$. 
Due to \tref{theorem}, it is enough to find a Lagrangian $\Lambda$ without
focal points on $I$.  By assumption, for each $a\in \R$  the space $\Lambda^a$
has no focal points with exception of $a$.  Let the time $a$ go to the a 
boundary of the interval $I$
and choose a convergent subsequence of the Lagrangian subspaces $\Lambda ^a$.
Then the limiting Lagrangian subspace  $\Lambda ^{\infty}$ 
(``the space  of parallel Jacobi fields'') has no focal points in $I$, due
to \pref{propo}.  
\end{proof}

\subsection{The main theorem} Now we are going to prove a slightly 
more general version of \tref{theorem}.

\begin{prop}  \label{gen} Let $V,I\subset \R,R,\Jac$ be as usual. Then for
any  Lagrangians $\Lambda _1, \Lambda _2 \subset \Jac$ and any 
interval $I_0$ we have  
$\ind _{\Lambda _1} (I_0) -\ind _{\Lambda _2} (I_0) 
\leq \dim (V) -\dim (\Lambda _1 \cap \Lambda _2)$.
\end{prop} 

\begin{proof} 
We proceed by induction on $\dim (V)$ and start with the case $\dim (V)=1$.
Then $\dim (\Lambda _i)=1$ and we may assume $\Lambda _1 \neq \Lambda _2$.
Assume that $\ind _{\Lambda _1} (I_0) -\ind _{\Lambda _2} (I_0 )\geq 2 $.

Note that for  any  $c\in[a,b]$ the  space $\Lambda ^c$ 
is $1$-dimensional, thus
all focal points have multiplicity one.  Therefore we find
an interval $I_1 \subset I_0$ with $\ind _{\Lambda _1 } (I_1) =2$
and $\ind _{\Lambda _2} (I_1)=0$. We may assume $I_1 =[a,b]$ and 
$\Lambda _1 =\Lambda ^a =\Lambda ^b$. 
Since $\dim (V)=1$,
the space $\Lagr$ of all Lagrangians is homeomorphic to $\R P^1 =S^1$. 
Consider the continuous 
map $F:I_1 \to \Lagr$ given by $F(c)=\Lambda ^c$. Due to \lref{dim1},
the map is locally injective, thus $F((a,b))$ is an open connected 
subset of $S^1$. Since $F(a)=F(b)=\Lambda _1$, the image $F(I_1)$ is a compact
subset of $S^1$ with at most one boundary point $\Lambda ^1$. But no compact
subset of $S^1$ has precisely one boundary point. Thus $F(I_1)=\Lagr$.
Therefore, there is some $c\in I_1$ with $\Lambda _2 =\Lambda ^c$.
This contradicts  $\ind _{\Lambda _2 } (I_1) =0$ and finishes the proof
in the case $\dim (V)=1$.

 Let us now assume $\dim (V)=m>1$ and let the result be true in all dimensions
smaller than $m$.  Consider the isotropic
subspace  $W=\Lambda _1 \cap \Lambda _2$ and assume that $W\neq 0$.
Then $\ind _{\Lambda _i} (I_0) = \ind _W (I_0) + \ind _{\Lambda _i /W} (I_0)$, 
for $i=1,2$. Thus  replacing $V,R$ by the $W$-transversal Jacobi
equation and using \lref{inddeco} and our inductive assumption  we get 
$$|ind _{\Lambda _1} (I_0) -ind _{\Lambda _2 } (I_0)| = 
|ind _{\Lambda _1 /W} (I_0) -ind _{\Lambda _2 /W} (I_0)| \leq 
\dim (V) -\dim (W)$$   
 This proves the statement in the case $W\neq 0$. In the  case $W=0$
one finds a Lagrangian $\Lambda _3$ with 
$\dim (\Lambda _3 \cap \Lambda _1)= m-1$ and 
$\dim (\Lambda _3 \cap \Lambda _2 )=1$ (To find  such $\Lambda _3$,
take any $(m-1)$-dimensional subspace $W$ of $\Lambda _1$, find a 
non-zero vector $J$ in the intersection $\Lambda _2 \cap W^{\perp}$
and set $\Lambda _3 := W \oplus \{ J \}$).  Using the result for
Lagrangians with non-zero intersection  we get:
$$|ind _{\Lambda _1} (I_0) -ind _{\Lambda _2 } (I_0)| \leq 
|ind _{\Lambda _1} (I_0) -ind _{\Lambda _3 } (I_0)| +
|ind _{\Lambda _1} (I_0) -ind _{\Lambda _2 } (I_0)| \leq m$$
\end{proof}

\bibliographystyle{alpha}
\bibliography{jacobi}

\begin{thebibliography}{FMT02}

\bibitem[Dui76]{Dui}
J.~Duistermaat.
\newblock On the {Morse} index in variational calculus.
\newblock {\em Advances in Math.}, 21:173--195, 1976.

\bibitem[EH90]{eschen}
J.-H. Eschenburg and E.~Heintze.
\newblock Comparison theory for {Riccati} equations.
\newblock {\em Manuscripta Math.}, 68:209--214, 1990.

\bibitem[MPT02]{piccione}
 F.~Mercuri, P.~Piccione  and D.~Tausk.
\newblock Stability of the conjugate index, degenerate conjugate points and the
  {Maslov} index in {semi-Riemannian} geometry.
\newblock {\em Pacific J. Math.}, 206:375--400, 2002.

\bibitem[LT07a]{Expl}
A.~Lytchak and G.~Thorbergsson.
\newblock Curvature explosion in quotients and applications.
\newblock Preprint; arXiv:math.DG/0709.2607, 2007.

\bibitem[LT07b]{LTvar}
A.~Lytchak and G.~Thorbergsson.
\newblock Variationally complete actions on nonnegatively curved manifolds.
\newblock {\em Illinois J. Math.}, 51:605--615, 2007.

\bibitem[O'N66]{o}
B.~O'Neill.
\newblock The fundamental equation of a submersion.
\newblock {\em Michigan J. Math.}, 13:459--469, 1966.

\bibitem[Sak96]{Sakai}
T.~Sakai.
\newblock {\em Riemannian geometry}.
\newblock Translations of Mathematical Monographs, 149. American Mathematical
  Society, Providence, RI, 1996.

\bibitem[Wil07]{Wilk}
B.~Wilking.
\newblock A duality theorem for {Riemannian} foliations in non-negative
  sectional curvature.
\newblock {\em Geom. Funct. Anal.}, 17:1297--1320, 2007.

\end{thebibliography}

\end{document}